\documentclass [11pt]{amsart}

\usepackage{calc}
\usepackage{amsmath,amssymb,epsfig}

\usepackage{amsmath,amsfonts,amssymb}
\usepackage{amsthm}
\usepackage{enumerate}
\usepackage{cancel}
\usepackage{multicol}
\usepackage{graphicx} 

\usepackage{subfigure}

\newtheorem{thm}{Theorem}[section]
\newtheorem{rem}[thm]{Remark}
\newtheorem{lem}[thm]{Lemma}
\newtheorem{prop}[thm]{Proposition}
\newtheorem{defn}[thm]{Definition}

   \newtheoremstyle{example}{\topsep}{\topsep}%
     {}
     {}
     {\bfseries}
     {}
     {\newline}
     {\thmname{#1}\thmnumber{ #2}\thmnote{ #3}}

   \theoremstyle{example}

   \newtheoremstyle{conjecture}{\topsep}{\topsep}%
     {}
     {}
     {\bfseries}
     {}
     {\newline}
     {\thmname{#1}\thmnumber{ #2}\thmnote{ #3}}

   \theoremstyle{conjecture}

   \newtheoremstyle{question}{\topsep}{\topsep}%
     {}
     {}
     {\bfseries}
     {}
     {\newline}
     {\thmname{#1}\thmnumber{ #2}\thmnote{ #3}}

   \theoremstyle{question}

\newcommand{\nR}{\mathbb R}
\newcommand{\nRR}{\mathbb R^4}
\newcommand{\nSS}{\mathbb S^4}

\newcommand{\cB}{\mathcal B}

\begin{document}
\title{Gluck twist on a certain family of 2-knots}

\author{Daniel Nash}
\address{
Alfr\'ed R\'enyi Institute of Mathematics,
Hungarian Academy of Sciences \newline
Re\'altanoda u. 13-15, 1053 Budapest, Hungary}
\email{dnash@renyi.hu}

\author{Andr\'as I. Stipsicz}
\address{
Alfr\'ed R\'enyi Institute of Mathematics,
Hungarian Academy of Sciences \newline
Re\'altanoda u. 13-15, 1053 Budapest, Hungary}
\email{stipsicz@renyi.hu}

\date{}

\begin{abstract}
We show that by performing the Gluck twist along the 2-knot $K^2_{pq}$ 
derived from two ribbon presentations of the ribbon 1-knot $K(p,q)$ we get 
the standard 4-sphere $\nSS$. In the proof we apply Kirby calculus.
\end{abstract}
\maketitle

\section{Introduction}
The paper of Freedman-Gompf-Morrison-Walker \cite{FGMW} about the
potential application of Khovanov homology in solving the 4-dimensional smooth 
Poincar\'e conjecture (SPC4) revitalized this important subfield of topology. A sequence
of papers appeared, some settling 30-year-old problems (\cite{Ak1, Gom1}), some
introducing new potential exotic 4-spheres (\cite{Nash}) and further works showing that
the newly introduced examples are, in fact, standard \cite{Ak2, Tange}. 

One underlying construction for producing examples of potential exotic 4-spheres is the Gluck 
twist along an embedded $S^2$ (a \emph{2-knot}) in the standard 4-sphere $\nSS$. In this construction we remove
the tubular neighbourhood of the 2-knot and glue it back with a specific diffeomorphism.
(For a more detailed discussion, see Section~\ref{sec:Gluckdefs}.)
 In turn, any 2-knot in $\nSS$ admits a 
\emph{normal form}, and hence can be described by an ordinary knot in $S^3$, together with 
two sets of ribbon bands (determining the `southern' and `northern' hemispheres of the
2-knot). Applying standard ideas of Kirby calculus (see, for example, \cite{GS99}) the
complement of a 2-knot, and from there the result of the Gluck twist, can be explicitly
drawn. From such a presentation then we derive the following result.

\begin{figure}[h]
\begin{center}
\includegraphics[height=6cm]{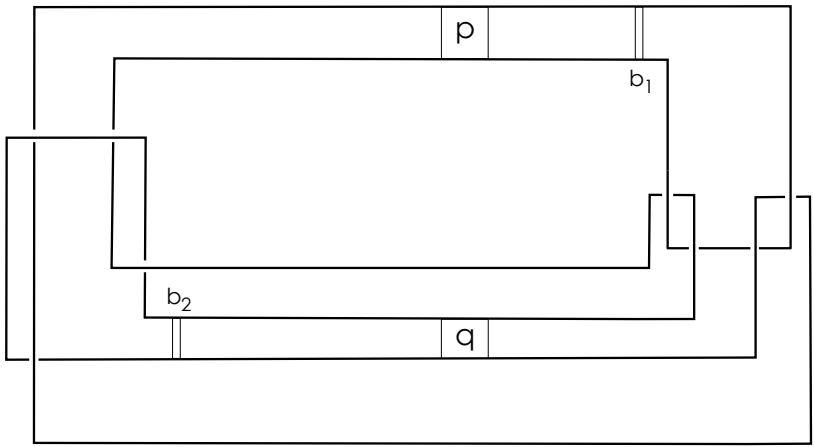}
\caption{The knot $K(p,q)$ with the two ribbon bands $b_1$ and $b_2$, giving rise to the
2-knot $K^2_{pq}\subset \nSS$.}
\label{fig:Kpqribbons}
\end{center}
\end{figure}
\begin{thm}\label{t:main}
Consider the knot $K(p,q)$ depicted by Figure~\ref{fig:Kpqribbons}, and use the bands $b_1$
and $b_2$ to construct the southern and northern hemispheres of a 2-knot
$K^2_{pq}\subset \nSS$. Then
Gluck twist along the 2-knot $K^2_{pq}$ provides the 4-sphere with its standard 
smooth structure.
\end{thm}

\begin{rem}
{\rm
For certain choices of $p$ and $q$ the 1-knot $K(p,q)$ can be identified more familiarly:  for instance, $K(0,0)$ is isotopic to $F\# F=F\# m(F)$, 
where $F$ is the Figure-8 knot (isotopic to its mirror image $m(F)$),
$K(1,-1)$ is the $8_9$ knot, while $K(1,1)$ is $10_{155}$ in the standard knot tables. Notice that in \cite{AK1}
the knot $8_9$ defines the 2-knot along which the Gluck twist is performed, 
although the bands used in \cite{AK1} are potentially different 
from $b_1$ and $b_2$ used in the theorem above, cf. \cite[Fig. 16]{AK1}.}
\end{rem}

Before we prove the above result, in Sections~\ref{sec:2knots} and 
\ref{sec:Gluckdefs}
we briefly invoke basic facts about 2-knots, the 
Gluck twist, and the derivation of a Kirby diagram for the result of the Gluck twist along 
a 2-knot given by a ribbon 1-knot and two sets of ribbons.
In Section~\ref{sec:proof} then a simple  Kirby calculus argument 
provides the proof of  Theorem~\ref{t:main}. (A slightly different argument,
still within Kirby calculus, for the same result is given in an Appendix.)

\bigskip

{\bf Acknowledgements}: The authors would like to acknowledge support by the 
\emph{Lend\"ulet} program of the Hungarian Academy of Sciences.
The second author was partially supported by OTKA Grant NK 81203.
 We also want to thank
Zolt\'an Szab\'o for numerous enlightening discussions.

\section{Ribbon 2-disks and related 2-knots}
\label{sec:2knots}

Every 2-knot is equivalent to one in normal form \cite{CKS04}, that is, for 
a 2-knot $K\subset \nSS$ there is an ambiently isotopic  $K' \approx S^2 \subset \nRR$ (i.e. $\nSS \setminus \infty)$ with a  projection $p: \nRR = \nR^3 \times \nR \rightarrow \nR$ such that $p$ restricted to $K'$ gives a Morse function with the properties:
\begin{enumerate}
\item $K' \subset \nR^3 \times [-c,c]$ some $c>0$,
\item all index-0 critical points are in $K' \cap \nR^3 \times \left\{-c\right\}$,
\item all index-1 critical points with negative $p$-value give fusion bands within $K' \cap \nR^3 \times (-c, 0)$,
\item $K' \cap \nR^3 \times \left\{0\right\}$ is a single 1-knot $k$,
\item all index-1 critical points with positive $p$-value
give fission bands within $K' \cap \nR^3 \times (0, c)$, and 
\item all index-2 critical points are in $K' \cap \nR^3 \times \left\{c\right\}$.
\end{enumerate}
In particular, this means that any 2-knot $K$ is formed from the union of two \emph{ribbon 2-disks} glued together along their boundaries which is the same (ribbon) 1-knot $k$ for both.  Since such a (ribbon disk) hemisphere $D$ of a 2-knot has a handlebody with only 0- and 1-handles, we can construct a Kirby diagram for any ribbon disk complement in the 4-disk $D^4$, and from there
for any 2-knot in $\nSS$ as follows (cf. \cite[Chapter~6]{GS99}). 

\begin{lem}\label{lem:1}
Let $K$ be a 2-knot given by the union of two ribbon disks with 
equatorial 1-knot $k$, 
lower hemisphere ribbon presentation $ \cB_1 = \left\{b_1, b_2, \dots, b_m \right\}$, and upper hemisphere ribbon presentation $\cB_2 = \left\{b_1 ', b_2 ', \dots, b_n ' \right\}$, as above.  
Then a handlebody for $\nSS \setminus K$ can be constructed by the 
following algorithm:
\begin{enumerate}
\item at each $\cB_1$ ribbon, split from $k$ into a dotted circle component and add a 2-handle as in the left diagram of Figure~\ref{fig:ribbonhandles},
\item at each $\cB_2$ ribbon, add the 2-handle as in the right diagram of 
Figure~\ref{fig:ribbonhandles}, and
\item add a 3-handle for each ribbon of $\cB_2$ and then a single 4-handle.
\end{enumerate}	
\end{lem}
\begin{figure}[h]
\begin{center}
\includegraphics[height=3cm]{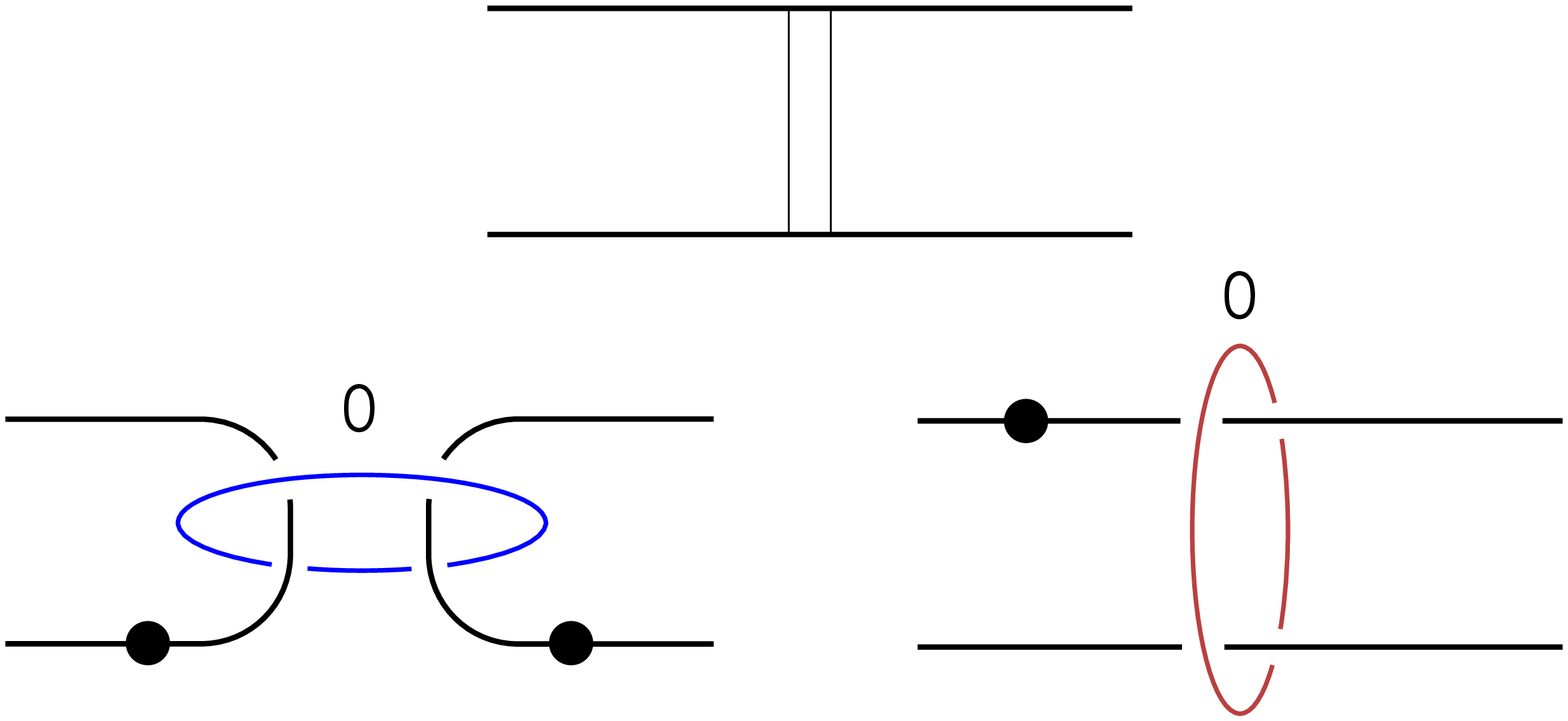}
\caption{Handles from a ribbon move in the lower hemisphere (left) and upper hemisphere (right).}
\label{fig:ribbonhandles}
\end{center}
\end{figure}

\begin{proof}
Let us start with describing the complement of one ribbon disk $D$; let  $X = D^4 \setminus D$ 
denote the ribbon disk complement.  Its handlebody starts with a 0-handle (four ball) $X_0$.  Then, for each 0-handle of $D$ which is carved out, a (4-dimensional) 1-handle is added to $X_0$ to form the 1-handlebody $X_1$.  Finally, the ribbons (or 1-handles) of $D$ each yield 2-handles in the complement, and these are attached along curves formed from the union of push-offs of the core 1-disk of the ribbons (cf. \cite[Section~6.2]{GS99}).  The attaching circles have 0-framings since push-offs of the core do not link. Therefore the result can be easily presented from a
diagram of the equatorial knot $k$ by locally replacing the bands with the 
diagram presented on the left of Figure~\ref{fig:ribbonhandles}.

Next consider the special case when $K$ is the 2-knot which we get by doubling the disk $D$, i.e. $K=D\cup \overline{D}$. In this case 
a Kirby diagram for the knot exterior $Y = \nSS - (D \cup \overline{ D})$ can be built up easily from the handlebody decomposition of $D$.  This amounts to taking the above disk complement $X$ and adding a second ``upside-down'' copy of $X$ (relative to the carved out 2-disk $D$, so that the result is still a manifold with boundary).  For each ribbon in the upper hemisphere, again we add a 2-handle to the complement.  However with $D$ ``turned upside-down'' in the upper hemisphere, the ribbons have cores and co-cores opposite to their counterparts in the lower hemisphere.  Consequently, the 2-handles added in the upper hemisphere's complement have attaching curves formed from the union of two co-cores of the original ribbons of $D$.  This is shown in figure \ref{fig:Doubledribbonhandles}, where a pair of 0-handles of $K$ and a 1-handle fusing them together gives the handlebody configuration in the complement on the right, with the vertical 2-handle depicting the ``upside-down'' copy coming from the upper hemisphere of $K$ ( and ``dotted circles'' depicting the 4-dimensional 1-handles).  Moreover, for each of the upper hemisphere 2-handles
(corresponding to 0-handles of $D$), we
get a 3-handle, and then finally a 4-handle to complete the description of
the complement of $K$.
For a similar discussion see \cite[Exercise~6.2.11(b)]{GS99}.

\begin{figure}[h]
\begin{center}
\includegraphics[width=7cm]{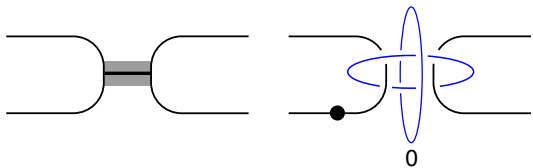}
\caption{Handles in $Y$ from a 1-handle in (each copy of) $D$.}
\label{fig:Doubledribbonhandles}
\end{center}
\end{figure}

Finally consider the general case, when the 2-knot $K$ is formed from
two disks $D_1$ and $D_2$ (as `southern' and `northern' hemispheres);
i.e. $K=D_1\cup \overline{D_2}$.
The recipe above provides a diagram for $D^4-D_1$ and for
$\nSS - (D_2\cup \overline{D_2})$. We only need to replace $D^4-D_2$
with $D^4-D_1$ to get a diagram for $\nSS -K$. This simply amounts to 
finding a diffeomorphism between the boundaries of $D^4-D_1$ and $D^4-D_2$,
and then pulling back the attaching circles of the 2-handles of $D^4-\overline{D_2}$
to the diagram of $D^4-D_1$.  By converting the dots to 0-framings, sliding the
(once dotted, now 0-framed) circles on each other and then canceling 
(in the 3-dimensional sense) the obvious handle pairs, we see that 
both $\partial (D^4-D_1)$ and $\partial (D^4-D_2)$ are diffeomorphic to the result of 
0-surgery
along the equatorial knot $k$. Using this diffeomorphism, the pull-back provides
the attaching circle given in the statement, concluding the proof.
\end{proof}

\section{The Gluck twist and Kirby diagrams}\label{sec:Gluckdefs}

Suppose that $K\subset \nSS$ is a given 2-knot in the 4-sphere.
Remove a normal neighborhood $\nu K$ of  $K \subset \nSS$ from the 4-sphere and reglue
 $S^2 \times D^2$ by the diffeomorphism of the boundary $\partial (S^2 \times D^2) \approx \partial (\nSS \setminus \nu K) \approx S^2 \times S^1$
$$\mu: S^2 \times S^1 \longrightarrow S^2 \times S^1,$$
given by $(x, \theta) \stackrel{\mu}{\mapsto} (rot_{\theta}(x), \theta)$ for $rot_{\theta}$ the rotation with angle $\theta$ of the 2-sphere about the axis through its poles.
\begin{defn}
{\rm
The above construction is called the \emph{Gluck twist}  along the 2-knot $K\subset 
\nSS$. 
The result of the Gluck twist along the 2-knot $K\subset \nSS$ will be denoted
by $\Sigma (K)$.}
\end{defn}
Since the result $\Sigma (K)$ of a Gluck twist is simply connected, Freedman's celebrated theorem implies that $\Sigma (K)$ is  homeomorphic to $\nSS$.

For a 2-sphere $K$ embedded in the 4-sphere, a handlebody for $\nu K$ consists of a 0-handle  plus one 2-handle attached along a 0-framed unknot.  This can also be built ``upside down'' from its boundary $S^2 \times S^1$ by attaching the (dualized) 2-handle $h_K$ along any meridian $\left\{{\mbox {pt.}}\right\} \times S^1$ of the sphere,  
and then attaching the dualized 0-handle as a 4-handle.  Therefore, if a handlebody diagram for the knot exterior $Y = \nSS \setminus \nu K$ is  given, then one can reconstruct $\nSS$ by attaching the 2-handle $h_K$ along a 0-framed meridian 
of any 1-handle $h$ corresponding to a 0-handle of $K$.  
The homotopy sphere $\Sigma(K)$  resulting from the Gluck twist on $K$ then can be formed from $Y$ by 
attaching the 2-handle $h_K$ with $\pm1$-framing along the same
meridional circle of the 1-handle $h$ (see also \cite[Exercise~6.2.4]{GS99}). Note that all the further attaching circles of 2-handles linking the 1-handle $h$ can be slid off $h$ by the use of 
$h_K$, and then $h$ and $h_K$ can be
cancelled against each other. Therefore, in practice the presentation of the Gluck twist along $K$
amounts to blowing down one of the dotted circles corresponding to a 0-handle of $K$
as if the dotted circle was a $(-1)$-framed (or a $(+1)$-framed, up to our choice) unknot. Notice also that in the preceding section we presented a diagram
for $Y$ which admits a 4-handle. Since in gluing $S^2\times D^2$ back we add a 
further 4-handle, one of them can be cancelled against a 3-handle.

In \cite{Go88} a further alternative of the effect of  the Gluck twist
is presented.  Since the rotation $rot_{\theta}$
involved in the gluing map fixes both poles $N$ and $S$ of the 2-sphere, there are two resulting fixed circles 
$\left\{N \right\} \times S^1$ and $\left\{S \right\} \times S^1$ of $\mu$. 
Presenting $S^2\times D^2$ as the union of a 0-handle, a 1-handle and two 2-handles
(or in the upside down picture two 2-handles $h_{K}$ and $h_{K}'$, a 3- and a 4-handle), one can  
construct $\Sigma(K)$ from $Y$ by attaching the two 2-handles 
$h_K$ and $h_K'$ (one along $\left\{N \right\} \times S^1$ and one along $\left\{S \right\} \times S^1$ with framings $(+1)$ and $(-1)$, respectively) and a 3- and 4-handle, where the two attaching circles are meridional circles of 
two dotted circles (corresponding to two 0-handles of $K$). Once again, since the 2-handles can be slid over $h_K$ and $h_K'$, and then these 2-handles can cancel the corresponding dotted circles, in 
practice the Gluck twist along $K$ amounts to simply blowing down 
 two dotted circles as if one were a $(+1)$-,
the other a $(-1)$-framed unknot (and then adding a 3- and a 4-handle). 
As before, one 3-handle cancels one of the two 4-handles appearing in the
decomposition.

In conclusion, if a 2-knot $K$ in normal form is given in $\nSS$ by a ribbon
knot $k$ with two sets of bands $\cB _1$ and $\cB _2$, 
then the above description provides a 
simple algorithmic way of producing a handle decomposition of the result 
$\Sigma (K)$ of the Gluck twist along $K$. Notice furthermore that if $\vert \cB _1\vert =1$
(or $\vert \cB _2\vert =1$) then the resulting decomposition can be chosen not to contain any
1-handles.

\begin{rem} 
{\rm
 For some special classes of 2-knots $K$, the diffeomorphism type of $\Sigma(K)$ is well-understood:  in \cite{Gordon} Gordon proved that $\Sigma(K)$ is diffeomorphic to $\nSS$ for any twist-spun 2-knot $K$.  Then in \cite{Melvin} Melvin showed that every ribbon 2-knot $K$ has $\Sigma(K)$ standard as well.  Additionally, since any ribbon 2-knot is the double of a ribbon 2-disk, \cite[Exercise 6.2.11(b)]{GS99} gives an alternate proof of this second result.}
\end{rem}

\section{A family of 2-knots}
\label{sec:proof}
For $p,q$ (possibly non-distinct) integers let $K(p,q)$ be the knot of
Figure~\ref{fig:Kpqribbons}.  This is a ribbon knot of 1-fusion, that is, there is a ribbon presentation of $K(p,q)$ such that performing the indicated single ribbon move transforms the knot into a two-component unlink.  In fact, there are two apparent choices for the single ribbon move (or two apparently distinct ribbon presentations).  These are indicated by the fine-lined bands $b_i$  ($i=1,2$) of
Figure~\ref{fig:Kpqribbons}.
Either of these ribbon presentations corresponds to a ribbon 2-disk which we will denote by $D(p,q)_i$ ($i=1,2$).  

\begin{defn}
{\rm 
Define the 2-knot $K^2_{pq}$ as the union $(D^4, D(p,q)_1) \cup \overline{(D^4, D(p,q)_2)}.$}
\end{defn}

Following the recipe of Section~\ref{sec:Gluckdefs} we exhibit a handlebody description of the 
result $\Sigma(K^2 _{pq})$ of the Gluck twist along $K^2_{pq}$. 
In doing so, first we present a diagram for $Y_{pq}=\nSS - \nu K^2 _{pq}$ in
Figure \ref{fig:KpqUnion} below:

\begin{figure}[h]
\begin{center}
\includegraphics[height=6cm]{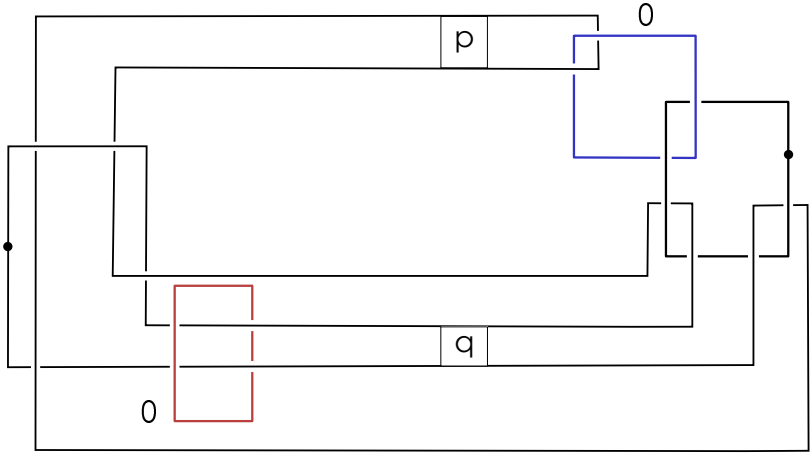}
\caption{Kirby diagram for $Y_{pq}$ minus two 3-handles and one 4-handle.}
\label{fig:KpqUnion}
\end{center}
\end{figure}

Figures~\ref{fig:1Iso1} through \ref{fig:Kpqpi1} demonstrate an isotopy of the above 
diagram of $Y_{pq}$ into a form where the 1-handles are visibly separated.
In Figure~\ref{fig:1Iso1} we have the result of undoing the $p$-twist in the first 1-handle and then starting to isotope the 2-handle through the second 1-handle.  
In Figure~\ref{fig:1Iso3}, the $q$-twist of the second 1-handle is undone by twisting the indicated four strands of the 2-handle.  
Further isotopies then finally produce Figure~\ref{fig:Kpqpi1}, where the
1-handles are conveniently separated.

\begin{figure}[h]
\begin{center}
\includegraphics[height=5cm]{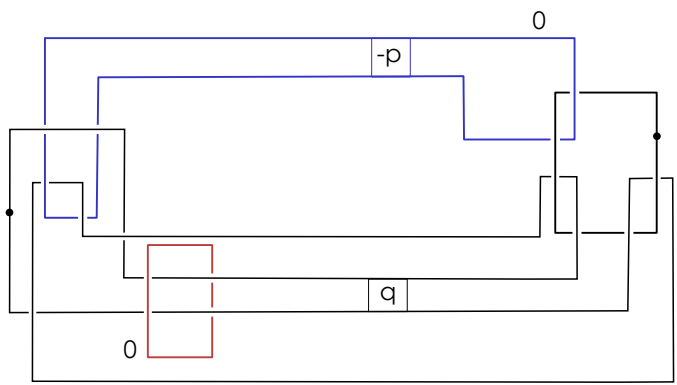}
\caption{Transferring the $p$-twist from the dotted circle to the 0-framed unknot.}
\label{fig:1Iso1}
\end{center}
\end{figure}


\begin{figure}[h]
\begin{center}
\includegraphics[height=5cm]{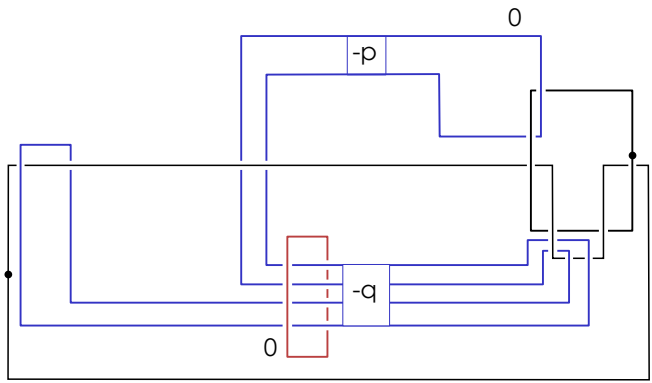}
\caption{A further isotopy of the diagram of Figure~\ref{fig:1Iso1}.}
\label{fig:1Iso3}
\end{center}
\end{figure}

\begin{figure}[h]
\begin{center}
\includegraphics[height=5cm]{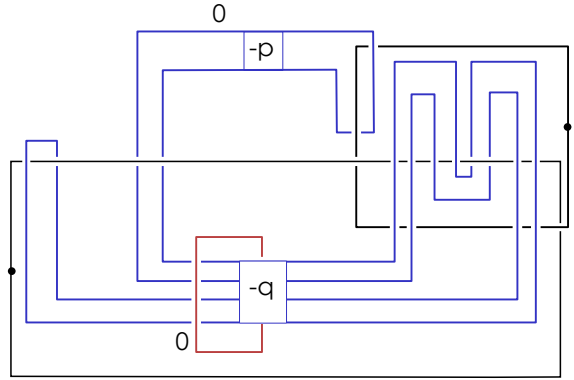}
\caption{Isotopy to separate the dotted circles.}
\label{fig:1Iso4}
\end{center}
\end{figure}


\begin{figure}[h]
\begin{center}
\includegraphics[height=5cm]{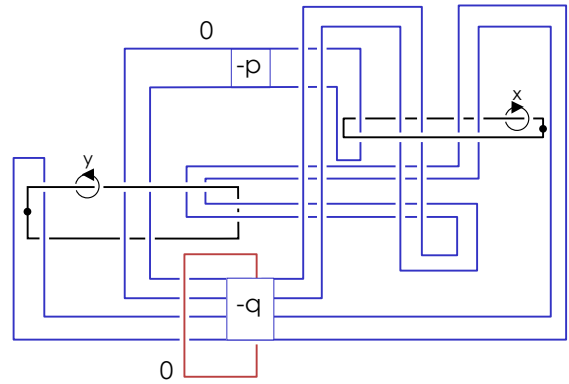}
\caption{The knot complement $Y_{pq}$ (minus two 3-handles and one 4-handle) with 1-handles separated (and generators of $\pi_1$ labeled).}
\label{fig:Kpqpi1}
\end{center}
\end{figure}

With this handlebody depiction of $Y_{pq}$ we can begin to analyze the 2-knot $K^2 _{pq}$, 
compute its knot group directly  and show the following:

\begin{prop}  The two 2-knots $K^2 _{pq}$ and 
$K^2 _{rs}$ are distinct provided that the parities of the pairs 
$\left\{p,q\right\}$, $\left\{r,s\right\}$ (up to reordering within a pair) are distinct.
\end{prop}

\begin{proof}
Choosing orientations on generators of $\pi_1$ as in Figure~\ref{fig:Kpqpi1}, we obtain $\pi_1(Y_{pq}) \cong \langle x,y\  |\  r_{pq} \rangle$, where the relation $r_{pq}$ takes one of the following four forms:
\begin{align*}
&\text{$p$ even, $q$ even}  \quad r_{pq}: \ xyxy^{-1}x^{-1}yxyx^{-1}y^{-1}\\
&\text{$p$ odd,  $q$ odd} \ \ \quad r_{pq}: \ xyxyx^{-1}y^{-1}xy^{-1}x^{-1}y^{-1}\\
&\text{$p$ odd,  $q$ even} \ \quad r_{pq}: \ xyxy^{-1}x^{-1}y^{-1}xyx^{-1}y^{-1}\\
&\text{$p$ even, $q$ odd} \ \quad r_{pq}: \ xyxyx^{-1}yxy^{-1}x^{-1}y^{-1}\\
\end{align*}

Using Fox calculus it is easy to see 
that each 2-knot does have a principal first elementary ideal and hence, an Alexander polynomial:
 
\begin{align*}
&\text{$p$ even, $q$ even} \quad \Delta(t) = -t^2 + 3t -1,\\
&\text{$p$ odd, $q$ odd} \ \ \quad \Delta(t) = 1- t + 2t^2 - t^3,\\
&\text{$p$ odd, $q$ even} \ \quad \Delta(t) = 2- 2t + t^2,\\
&\text{$p$ even, $q$ odd} \ \quad \Delta(t) = 2t^2 - 2t + 1.\\
\end{align*}
This gives three clearly distinguished cases for a pair $\left\{p,q\right\}$.  In particular, $K^2 _{pq}$ and $K^2 _{rs}$ have distinct Alexander polynomials if the pairs of parities are distinct.
\end{proof}

\begin{rem}  
{\rm In the three cases other than $p,q$ both even, $\Delta_K$ is asymmetric and therefore it is not the Alexander polynomial of a 1-knot.  Consequently, $K^2 _{pq}$ cannot be a spun knot if $p,q$ are not both even.}
\end{rem}

Now we are ready to provide a proof of the main result of the paper:

\begin{proof}[Proof of Theorem~\ref{t:main}]
 After isotoping the small 2-handle until it becomes parallel to the rightmost
 dotted circle and blowing down the two dotted circles (as the implementation
 of the Gluck twist demands) we arrive at Figure~\ref{fig:bdGluck}.
 Finally, figure \ref{fig:Gluckfree} (obtained by performing the indicated handle slide in \ref{fig:bdGluck}) unravels to give a pair of disjoint 0-framed 2-handles which cancel against the 3-handles to give $\nSS$.
\end{proof}



\begin{figure}[h]
\begin{center}
\includegraphics[height=5cm]{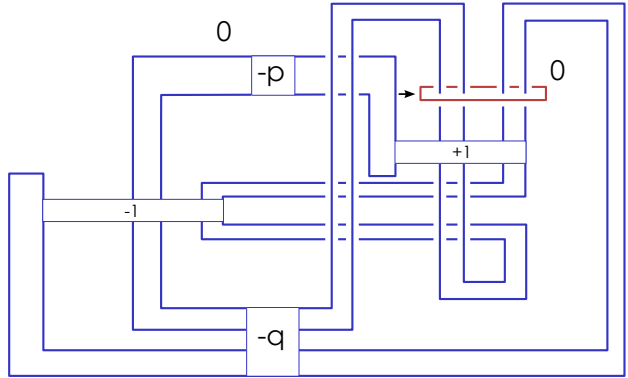}
\caption{The diagram of $\Sigma (K^2 _{pq})$ (minus two uniquely attached 3-handles and one 4-handle).}
\label{fig:bdGluck}
\end{center}
\end{figure}

\begin{figure}[h]
\begin{center}
\includegraphics[height=5cm]{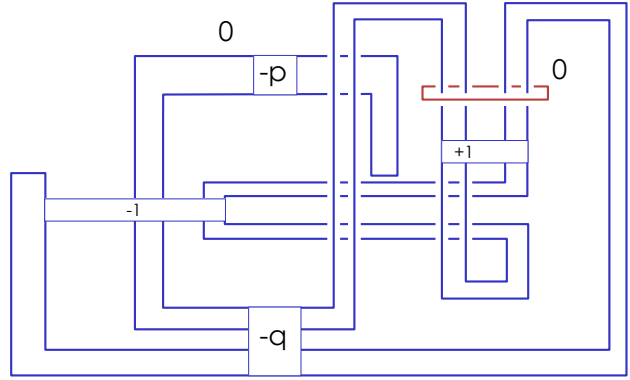}
\caption{After sliding one 2-handle over the other one (as instructed by 
the arrow in Figure~\ref{fig:bdGluck}), we are left with two unlinked 2-handles,
which can be seen to cancel against the 3-handles.}
\label{fig:Gluckfree}
\end{center}
\end{figure}

\section*{Appendix:  Alternate proof of Theorem~\ref{t:main}}
For the particular case of the 2-knot $K^2 _{pq}$ there is, in fact, a way to see that the Gluck twist leaves $\nSS$ standard without separating the 1-handles first.
\begin{proof}
Starting in $Y_{pq}$ (cf. Figure~\ref{fig:KpqUnion}), realize the Gluck twist on $K^2 _{pq}$ by adding a $(-1)$-framed 2-handle to a 1-handle (instead of just immediately blowing down the dotted 1-handle).
Slide the lower 0-framed 2-handle over the upper 1-handle and off of  the $q$-twisted 
1-handle to get Figure~\ref{1stmove}.  Now in Figure~\ref{1stmove} slide the 
$(-1)$-framed 2-handle over the left-most 0-framed 2-handle and off of its 1-handle and the 
right-most 2-handle.  Next, in Figure~\ref{fig:3rdmove} slide the 0-framed 2-handle on the left over the 
$(-1)$-framed 2-handle (which changes its own framing to $-1$) 

and use the remaining 0-framed 2-handle 
to unhook the
other two 2-handles from each other. This results in a collection of
Hopf links, and standard handle cancellations then show that 
$\Sigma (K^2 _{pq})$ is, indeed, diffeomorphic to the standard 4-sphere $\nSS$.

\end{proof}



\begin{figure}[h]
\begin{center}
\includegraphics[height=5cm]{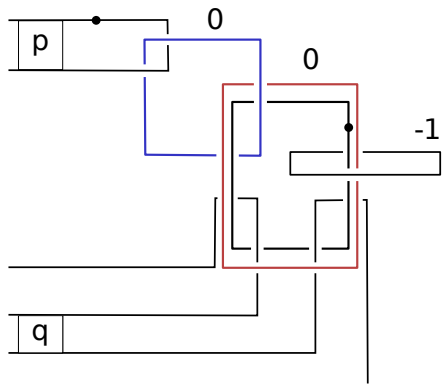}
\caption{The relevant portion of the diagram of $\Sigma (K^2_{pq})$ (again, minus 3- and 4-handles) after one handle slide.}
\label{1stmove}
\end{center}
\end{figure}

\begin{figure}[h]
\begin{center}
\includegraphics[height=5cm]{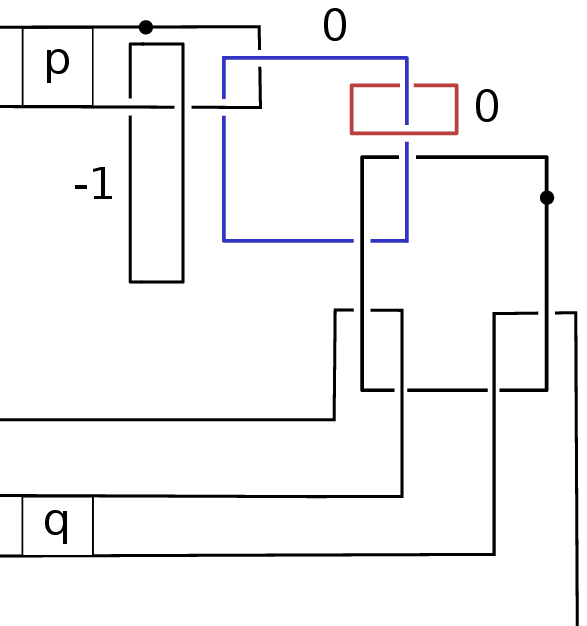}
\caption{The relevant portion of the diagram after sliding the $(-1)$-framed 2-handle and isotoping the right-most 0-framed component.}
\label{fig:3rdmove}
\end{center}
\end{figure}




\bibliographystyle{abbrv}


\end{document}